\documentclass[twoside,leqno,10pt]{amsart}
\usepackage{amsfonts}
\usepackage{amsmath}
\usepackage{amscd}
\usepackage{amssymb}
\usepackage{amsthm}
\usepackage{amsrefs}
\usepackage{latexsym}
\usepackage{bbm}
\usepackage{color}
\begin{document}
\newtheorem{theorem}{Theorem}
\newtheorem{lemma}{Lemma}
\newtheorem{prop}{Proposition}
\newtheorem{corollary}{Corollary}
\newtheorem{conjecture}{Conjecture}
\numberwithin{equation}{section}
\newcommand{\dif}{\mathrm{d}}
\newcommand{\intz}{\mathbb{Z}}
\newcommand{\ratq}{\mathbb{Q}}
\newcommand{\natn}{\mathbb{N}}
\newcommand{\comc}{\mathbb{C}}
\newcommand{\rear}{\mathbb{R}}
\newcommand{\prip}{\mathbb{P}}
\newcommand{\uph}{\mathbb{H}}

\title{A Lower Bound for the Large Sieve with Square Moduli}
\author{Stephan Baier \and Sean B. Lynch \and Liangyi Zhao}
\date{\today}

\begin{abstract}
We prove a lower bound for the large sieve with square moduli.
\end{abstract}

\maketitle



\section{Introduction}

The classical large sieve inequality states that for $Q, N \in \natn$, $M\in \intz$ and any sequence of complex numbers 
$\{a_n\}$,
\[ \sum_{q=1}^Q \sum_{\substack{a=1 \\ (a,q)=1}}^q \left| \sum_{n=M+1}^{M+N} a_n e \left( \frac{an}{q} \right) \right|^2 \leq (Q^2+N-1) \sum_{n=M+1}^{M+N} |a_n|^2 . \]

In \cite{Zha}, the third author studied the large sieve inequality for square moduli and conjectured that for
any $\varepsilon>0$,
\begin{equation} \label{conj}
  \sum_{q=1}^Q \sum_{\substack{a=1 \\ (a,q)=1}}^{q^2} \left| \sum_{n=M+1}^{M+N} a_n e \left( \frac{an}{q^2} \right) 
  \right|^2 \ll Q^{\varepsilon} (Q^3+N) \sum_{n=M+1}^{M+N} |a_n|^2,
\end{equation}
where the implied constant depends only on $\varepsilon$. In his undergraduate thesis, 
the second author numerically investigated the validity of \eqref{conj}.  
A natural question is whether \eqref{conj} can hold with the factor 
$Q^{\varepsilon}$ removed. In this note,
we answer this question in the negative. More precisely, we prove the following.

\begin{theorem} \label{main} For every $\varepsilon>0$, there are infinitely many natural numbers $Q$ such that for 
suitable $M\in \mathbb{Z}$, $N\in \mathbb{N}$ and sequences
$\{a_n\}$ of complex numbers, we have
\begin{equation} \label{goal}
\sum_{q=1}^Q \sum_{\substack{a=1 \\ (a,q)=1}}^{q^2} \left| \sum_{n=M+1}^{M+N} 
a_n e \left( \frac{an}{q^2} \right) \right|^2 \ge DQ^{\frac{\log 2}{(1+\varepsilon)\log \log Q}}
(Q^3+N) \sum_{n=M+1}^{M+N} |a_n|^2
\end{equation}
for some absolute positive constant $D$. 
\end{theorem}

The above theorem shows that the $Q^{\varepsilon}$ factor in \eqref{conj} cannot be discarded or even replaced by a power of logarithm.  We note that the best-known upper bound for the left-hand side of \eqref{conj} is
$$
\ll (QN)^{\varepsilon}\left(Q^3+N+\min\left\{\sqrt{Q}N,\sqrt{N}Q^2\right\}\right) \sum_{n=M+1}^{M+N} |a_n|^2
$$
due to the first and third authors \cite{SBLZ2}. \newline

The large sieve inequality for square (and quadratic) moduli has many applications. For example, it is used in the study of the Bombieri-Vinogradov theorem for square moduli \cite{baier-zhao3}, elliptic curves over finite fields \cites{banks-pappalardi-shparlinski, ISLZ}, Fermat quotients \cite{bourgain-ford-konyagin-shparlinski}, and the representation of primes \cites{baier-zhao3, matomaki}. \newline

In \cite{Zha}, the third author also studied the large sieve inequality for $k$-power moduli, where $k>2$.  The best known result for these $k$-power moduli with $k >2$ is due to K. Halupczok \cite{Hal}, who gave a large sieve inequality for $k$-power moduli which is uniform in $k$. \newline

{\bf Acknowledgement.} The first author would like to thank the University of New South Wales (UNSW) for its financial support and hospitality during his pleasant stay at the School of Mathematics and Statistics. The second author thanks the following organisations for his Research Training Program scholarship: the Department of Education and Training, Australian Government and the School of Mathematics and Statistics, UNSW.  The third author was supported by the FRG grant PS43707 and the Faculty Silverstar Fund PS49334 at UNSW during this work.

\section{Proof of Theorem \ref{main}}
We first establish the following lower bound for the number of Farey fractions with square denominators near 
certain rational points. 

\begin{lemma} \label{Farey} Let $\varepsilon>0$ and $p_1,...,p_m$ be the first $m$ odd primes. 
Set $Q:=p_1\cdots p_m$ and 
\begin{equation} \label{SQ}
\mathcal{S}(Q):=\left\{ (a,q)\in \mathbb{N}\times \mathbb{N} : Q<q\le 2Q, \ 1\le a\le q^2,\ (a,q)=1,\ \left|\frac{a}{q^2}-
\frac{1}{Q}\right| \le \frac{1}{Q^3} \right\}.
\end{equation}
Then
\begin{equation} \label{count}
\sharp\mathcal{S}(Q)\ge Q^{\frac{\log 2}{(1+\varepsilon)\log \log Q}},  
\end{equation}
provided $m$ is sufficiently large.
\end{lemma}

Here we note that the expected number of Farey fractions of the form $a/q^2$ with $Q<q\le 2Q$, $1\le a\le q^2$, 
$(a,q)=1$ in an interval of length $\Delta$ is, heuristically, of order of magnitude $Q^3\Delta$. So the above Lemma \ref{Farey}
shows that under certain circumstances, the true number can exceed the expectation significantly. \\

\begin{proof}[Proof of Lemma \ref{Farey}] Using the Chinese Remainder Theorem, the number of solutions to the congruence
 $$
 q^2\equiv 1 \pmod{Q}
 $$
 with $Q<q\le 2Q$ is exactly $2^m$. If $q$ solves the above congruence, then
 $$
 q^2=1+aQ
 $$
 for some $a$ with $1\le a\le q^2$ and $(a,q)=1$, and it follows that
 $$
 \left|\frac{a}{q^2}-\frac{1}{Q}\right| = \frac{1}{q^2Q}\le \frac{1}{Q^3}.
 $$
 Hence,
 $$
 \sharp\mathcal{S}(Q)\ge 2^m.
$$
Moreover, using the prime number theorem, for any given $\varepsilon>0$, 
$$
\log Q = \sum\limits_{i=1}^m \log p_i \le (1+\varepsilon)p_m \le (1+2\varepsilon)m\log m
$$
if $m$ is sufficiently large. Consequently, for any given $\varepsilon>0$, 
$$
m\ge \frac{\log Q}{(1+\varepsilon)\log \log Q}
$$
if $m$ is large enough. Now the desired inequality \eqref{count} follows.
\end{proof}

Having proved Lemma \ref{Farey}, we are ready to prove Theorem \ref{main}. It will suffice to prove
\eqref{goal} with the summation range $1\le q\le Q$ replaced by $Q<q\le 2Q$ which we do in the following. 
Set $Q=p_1\cdots p_m$ as in 
Lemma \ref{Farey}. Further, set 
$$
M:=0, \quad N:=\frac{Q^3}{9}, \quad 
a_n: = e\left(-\frac{n}{Q}\right).
$$
Then
$$
\sum_{n=M+1}^{M+N} a_n e \left( \frac{an}{q^2} \right) = 
\sum_{n=1}^{N} e \left( \alpha_n\right)
$$
with 
$$
\alpha_n:=n\left(\frac{a}{q^2}-\frac{1}{Q}\right).
$$
If 
$$
\left|\frac{a}{q^2}-\frac{1}{Q}\right| \le \frac{1}{Q^3},
$$
then $|\alpha_n|\le 1/9$ for $n=1,...,N$ and hence
\begin{equation} \label{nocancel}
\left|\sum_{n=1}^{N} e \left( \alpha_n\right)\right| \ge CN
\end{equation}
for some absolute positive constant $C$. \newline

Define $\mathcal{S}(Q)$ as in \eqref{SQ}.
Then we have
\begin{equation*}
\begin{split}
& \sum_{q=Q+1}^{2Q} \sum_{\substack{a=1 \\ (a,q)=1}}^{q^2} \left| \sum_{n=M+1}^{M+N} 
a_n e \left( \frac{an}{q^2} \right) \right|^2 \\ \ge &  
\sum_{(a,q)\in \mathcal{S}(Q)} \left| \sum_{n=M+1}^{M+N} 
a_n e \left( \frac{an}{q^2} \right) \right|^2 \\ \ge & \sharp \mathcal{S}(Q)\cdot (CN)^2\\ = &
C^2\cdot \sharp \mathcal{S}(Q) \cdot N \sum\limits_{n=M+1}^{M+N} |a_n|^2\\ 
= & \frac{C^2}{10} \cdot \sharp \mathcal{S}(Q) \cdot (Q^3+N) \sum\limits_{n=M+1}^{M+N} |a_n|^2\\ 
\ge & \frac{C^2}{10} \cdot Q^{\frac{\log 2}{(1+\varepsilon)\log \log Q}} \cdot (Q^3+N) \sum\limits_{n=M+1}^{M+N} |a_n|^2,
\end{split}
\end{equation*}
where the third line follows from \eqref{nocancel}, and the last line follows from
Lemma \ref{Farey}. This completes the proof. 

\bibliography{biblio.bib}

\begin{bibdiv}
\begin{biblist}

\bib{baier-zhao3}{article}{
      author={Baier, S.},
      author={Zhao, L.},
       title={Bombieri-{V}inogradov theorem for sparse sets of moduli},
        date={2006},
     journal={Acta Arith.},
      volume={125},
      number={2},
       pages={187\ndash 201},
}

\bib{SBLZ2}{article}{
      author={Baier, S.},
      author={Zhao, L.},
       title={An improvment for the large sieve for square moduli},
        date={2008},
     journal={J. Number Theory},
      volume={128},
      number={1},
       pages={154\ndash 174},
}

\bib{banks-pappalardi-shparlinski}{article}{
      author={Banks, W.~D.},
      author={Pappalardi, F.},
      author={Shparlinski, I.~E.},
       title={On group structures realized by elliptic curves over arbitrary
  finite fields},
        date={2012},
     journal={Experimental Mathematics},
      volume={21},
      number={1},
       pages={11\ndash 25},
}

\bib{bourgain-ford-konyagin-shparlinski}{article}{
      author={Bourgain, J.},
      author={Ford, K.},
      author={Konyagin, S.~V.},
      author={Shparlinski, I.~E.},
       title={On the divisibility of {F}ermat quotients},
        date={2010},
     journal={Michigan Math. J.},
      volume={59},
       pages={313\ndash328},
}

\bib{Hal}{article}{
      author={Halupczok, K.},
       title={A new bound for the large sieve inequality with power moduli},
        date={2012},
     journal={Int. J. Number Theory},
      volume={8},
      number={3},
       pages={689\ndash 695},
}

\bib{matomaki}{article}{
      author={Matom\"aki, K.},
       title={A note on primes of the form $p= aq^2+ 1$},
        date={2009},
     journal={Acta Arith.},
      volume={137},
      number={2},
       pages={133\ndash137},
}

\bib{ISLZ}{article}{
      author={Shparlinski, I.~E.},
      author={Zhao, L.},
       title={Elliptic curves in isogney classes},
        date={2018},
     journal={J. Number Theory},
      volume={191},
       pages={194\ndash 212},
}

\bib{Zha}{article}{
      author={Zhao, L.},
       title={Large sieve inequality for characters to square moduli},
        date={2004},
     journal={Acta Arith.},
      volume={112},
      number={3},
       pages={297\ndash 308},
}

\end{biblist}
\end{bibdiv}
\bibliographystyle{amsxport}

\vspace*{.7cm}

\noindent Stephan Baier, Department of Mathematics, RKMVERI \newline
G.T. Road, Belur Math, Howrah, West Bengal, India-711202 \newline
Email: {\tt stephanbaier2017@gmail.com} \newline

\noindent Sean B. Lynch, School of Mathematics and Statistics, University of New South Wales \newline
UNSW-Sydney NSW 2052, Australia \newline
Email: {\tt s.b.lynch@unsw.edu.au} \newline

\noindent Liangyi Zhao, School of Mathematics and Statistics, University of New South Wales \newline
UNSW-Sydney NSW 2052, Australia \newline
Email: {\tt l.zhao@unsw.edu.au}
\end{document}